\documentclass[11pt]{amsart}
\usepackage{latexsym}
\usepackage{amsfonts,amsmath,amssymb,stmaryrd,amscd,amsxtra,amsthm,verbatim,calc}
\usepackage[all]{xy}

\usepackage{xypic}
\usepackage{eucal}

\usepackage{amsmath}
\usepackage{amstext}
\usepackage{amssymb}
\usepackage{amsthm}
\usepackage{amscd}

\usepackage{lscape}
\usepackage{longtable}

\usepackage{epsfig}
\input txdtools
\let\et=\etexdraw
\def\etexdraw{\drawbb\et}

\theoremstyle{plain}
\newtheorem{thm}{Theorem}[section]
\newtheorem{thm*}{Theorem}
\newtheorem{lem}[thm]{Lemma}
\newtheorem{prop}[thm]{Proposition}
\newtheorem{prop*}[thm*]{Proposition}
\newtheorem{cor}[thm]{Corollary}

\theoremstyle{definition}
\newtheorem{defi}[thm]{Definition}

\newtheorem{notation}[thm]{Notation}

\theoremstyle{remark}

\DeclareMathOperator{\Coker}{Coker}
\DeclareMathOperator{\Image}{Im}

\DeclareMathOperator{\height}{ht}
\DeclareMathOperator{\Hom}{Hom}

\DeclareMathOperator{\Ext}{Ext}

\DeclareMathOperator{\Ann}{ann}
\DeclareMathOperator{\Nil}{Nil}

\DeclareMathOperator{\HH}{H}

\DeclareMathOperator{\Supp}{Supp}
\DeclareMathOperator{\Spec}{Spec}

\DeclareMathOperator{\fm}{\mathfrak{m}}

\DeclareMathOperator{\fc}{\mathfrak{c}}

\begin{document}

\title{Global parameter test ideals}

\author{Mordechai Katzman}
\email{M.Katzman@sheffield.ac.uk}
\address{Department of Pure Mathematics,
University of Sheffield, Hicks Building, Sheffield S3 7RH, United Kingdom}

\author{Serena Murru}
\email{pmp11sm@sheffield.ac.uk}
\address{Department of Pure Mathematics,
University of Sheffield, Hicks Building, Sheffield S3 7RH, United Kingdom}

\author{Juan D.~Velez}
\email{jdvelez@unal.edu.co}
\address{Escuela de Matemáticas, Universidad Nacional de Colombia, Medellin,
Calle 59A No.~63 - 20,
Medellin, Antioquia, Colombia}

\author{Wenliang Zhang}
\email{wlzhang@uic.edu}
\address{Department of Mathematics, Statistics, and Computer Science, University of Illinois at Chicago,
851 S.~Morgan Street, Chicago, IL 60607-7045}


\thanks{M.~K.~gratefully acknowledges support from EPSRC grant EP/J005436/1. W.~Z.~is partially supported by NSF grant DMS \#1606414.}


\begin{abstract}
This paper shows the existence of
ideals whose localizations and completions at prime ideals are parameter test ideals of the
localized and completed rings. We do this for Cohen-Macaulay localizations (resp., completions) of non-local rings,
for generalized Cohen-Macaulay rings, and for non-local rings with isolated non Cohen-Macaulay points, each being an isolated non
$F$-rational point. The tools used to prove these results are constructive in nature and as a consequence our results
yield algorithms for the computation of these global  parameter test ideals.

Finally, we illustrate the power of our methods by analyzing the HSL numbers of local cohomology modules
with support at any prime ideal.
\end{abstract}

\maketitle


\section{Introduction}\label{Section: Introduction}

This paper studies certain properties of commutative rings of prime characteristic $p$.
Such rings $A$ are equipped with \emph{Frobenius maps} $f_e: A \rightarrow A$ defined as $f_e(a)=a^{p^e}$
and these give a good handle on various problems which are not available
in characteristic zero. One such handle is provided by the machinery of \emph{tight closure}
introduced by Mel Hochster and Craig Huneke in the 1990's (cf.~\cite{HochsterHunekeTC1}) which we now review.

The tight closure of an ideal $I\subseteq A$ is the set of all $a\in A$
for which for some element $c\in A$ not in any minimal prime
one has $c a^{p^e} \in I^{[p^e]}$ for all $e\gg 0$,
where $I^{[p^e]}$ denotes the ideal of $A$ generated by all $p^e$th powers of elements in $I$.


One of the basic properties of tight closure is that for ideals $I$ of regular rings one has $I^*=I$ (we refer to ideals with this property as being \emph{tightly closed}),
or equivalently, that $c=1$ and $e\geq 0$ can be used in the definition above to
test membership in the tight closure of ideals. This suggests that measuring the failure of the tight closure operation to be trivial might be a useful way of measuring how bad the
singularity of a ring is, and one thus obtains a hierarchy of singularities:
\begin{enumerate}
\item[(a)] regular rings,
\item[(b)] \emph{$F$-regular} rings: all ideals in all localizations are tightly closed,
\item[(c)] \emph{weakly $F$-regular} rings: all ideals  are tightly closed,
\item[(d)] \emph{$F$-rational} rings: all ideals generated by parameters are tightly closed.
\end{enumerate}
From this point of view, the object of interest is the set of elements $c$ in the definition of tight closure
which can be used to test membership in the tight closure of ideals.

\begin{defi}{({\it cf.} \cite[\S6]{HochsterHunekeTC1} and \cite[\S8]{HochsterHunekeTightClosureParameter})}
An element $c$ not in any minimal prime is a \emph{test element} if for all ideals $I$ and all $e\geq 0$, $c{I^*}^{[p^e]}\subseteq I^{[p^e]}$.
The \emph{test ideal} is defined as the ideal generated by all test elements.

An element $c$ not in any minimal prime is a \emph{parameter test element} if for all ideals $I$ generated by parameters and all $e\geq 0$, $c{I^*}^{[p^e]}\subseteq I^{[p^e]}$.
The \emph{parameter test ideal} is the ideal generated by all parameter test elements.
\end{defi}
Thus properties (c) and (d) above can be restated as the test ideal and the parameter test ideal being the unit ideal, respectively.

In this paper we construct \emph{global parameter test ideals} of finitely generated algebras, i.e.,
ideals whose localization are parameter test ideals of the localized rings. Among other things, our explicit description of
these ideals yields an explicit description of the $F$-rational locus of finitely generated algebras, recovering the fact that the $F$-rational locus is open
(cf.~\cite{VelezOpennessOfTheFRationalLocus}),
and  in the process also providing a method for computing global parameter test ideals.

\section{Prime characteristic tools}\label{Section: Prime characteristic tools}

In this section we introduce various tools and notation used to study rings of prime characteristic and their modules.  

We start with the elementary observation that if $A$ is a ring of prime characteristic $p$, the \emph{$e$th iterated Frobenius map} $f^e: A \rightarrow A$
taking $a\in A$ to $a^{p^e}$ ($e\geq 0$)
is a homomorphism of rings.
The usefulness of these homomorphisms lies in the fact that given an $A$-module $M$, we may endow it with a new $A$-module structure via $f^e$:
let  $F^e_* M$ denote the additive Abelian group $M$,
denoting its elements $\{ F^e_* a \,|\, a\in M \}$, and endow $F^e_* M$  with the $A$-module structure  given by  $a F_*^e m = F^e_* a^{p^e} m$.

Given any $A$-linear map $g: M \rightarrow F_*^e M$, we have an additive map  $\widetilde{g} : M \rightarrow M$ obtained by identifying $F_*^e M$ with $M$.
This map $\widetilde{g}$ is not $A$-linear: it satisfies $\widetilde{g}(a m)=a^{p^e} \widetilde{g}(m)$ for all $a\in A$ and $m\in M$.
We call additive maps with this property \emph{$e$th Frobenius maps}.
Conversely, an $e$th Frobenius map $h: M \rightarrow M$ defines an $A$-linear map  $M \rightarrow F_*^e M$ given by $m\mapsto F_*^e h(m)$.

A convenient way to keep track of Frobenius maps is provided by the use of certain skew-polynomial rings, defined as follows.
Let $A[\Theta; f^e]$ be the free $A$-module $\bigoplus_{i\geq 0} A \Theta^i$ and
give $A[\Theta; f^e]$ the structure of a ring by defining the (non-commutative) product $(a \Theta^i) (b \Theta^j) = a b^{p^{e i}} \Theta^{i+j}$.
(This is an instance of a skew-polynomial ring: cf.~\cite[Chapter 1]{LamAFirstCourseInNoncommutativeRings}.)
An $A$-module $M$ equipped with an $e$th Frobenius map $\widetilde{g}$ is nothing but an
$A[\Theta; f^e]$-module where the action of $\Theta$ on $M$ is given by $\Theta m = \widetilde{g}(m)$ for all $m\in M$.

A crucial set of tools in the prime characteristic toolkit are the \emph{Frobenius functors} which we now define.
For any $A$-module $M$ and $e\geq 1$, we can extend scalars and obtain the $F_*^e A$-module  $F_*^e A \otimes_A M$. If we now identify the rings $A$ and
$F_*^e A$ we obtain the $A$-module $A \otimes_A M$ where
for $a, b\in A$, and $m\in M$ $a(b\otimes m)=ab\otimes m$ and $a^{p^e} b \otimes m=b\otimes a m$ and we denote this by $F^e_A(M)$.
Clearly, homomorphisms $M\rightarrow N$ induce $A$-linear maps $F_A^e(M) \rightarrow F_A^e(M)$ and thus we obtain the aforementioned
\emph{$e$th Frobenius functors.}

An $e$th Frobenius map $g: M \rightarrow M$
gives rise to the $A$-linear map $1\otimes g : F_A^e (M) \rightarrow M$;
this is well defined since for all $a, b \in A$ and $m\in M$
$$(1\otimes g) (a^{p^e} b\otimes m) = a^{p^e} b g(m) =  b g(a m) =  1\otimes g (b \otimes am) .$$
In the special case when $M$ is an Artinian module over a complete regular ring, this gives a way to define a
``Matlis-dual which keeps track of Frobenius'' functor we describe in the rest of this section. This functor and its inverse were introduced in \cite{KatzmanParameterTestIdealOfCMRings}
and used to study parameter test ideals in local Cohen-Macaulay rings.

Throughout the rest of this section we adopt the following notation.

\begin{notation}\label{Notation: complete regular local}
Let $(R, \mathfrak{m})$  denote a $d$-dimensional complete regular local ring of prime characteristic $p$ and $S$
its quotient by an ideal $I\subset R$.
We will denote $E=E_R(R/\mathfrak{m})$ and $E_S=E_S(S/\mathfrak{m}S)=\Ann_E I$
the injective hulls of the residue fields of $R$ and $S$, respectively.
The Matlis dual functor $\Hom_R (-, E)$ will be denoted $(-)^\vee$.
\end{notation}

A crucial ingredient for the construction that follows is the fact that for both Artinian and Noetherian modules $R$-modules $M$,
there is a natural identification of $F_R^e(M)^\vee$ with $F_R^e(M^\vee)$
(this is proved for Artinian modules in \cite[Lemma 4.1]{LyubeznikFModulesApplicationsToLocalCohomology} and a similar proof
applied to a presentation of a finitely generated $M$ yields this result for Notherian modules)
and  henceforth we use this identification tacitly.
In what follows we will consider  $R[\Theta; f^e]$-modules which are Artinian as $R$-modules and we
will refer to these as being Artinian $R[\Theta; f^e]$-modules.
A map of Artinian $R[\Theta; f^e]$-modules $\rho: M\rightarrow N$,
yields a commutative diagram
\begin{equation*}
\xymatrix{
F_R^e(M) \ar@{>}[r]^{F_R^e (\rho)} \ar@{>}[d]^{1\otimes \Theta} & F_R^e(N) \ar@{>}[d]^{1\otimes \Theta} \\
M \ar@{>}[r]^{\rho} & N \\
}
\end{equation*}
and an application of the Matlis dual gives the commutative diagram
\begin{equation*}
\xymatrix{
N^\vee \ar@{>}[r]^{\rho^\vee} \ar@{>}[d]^{1\otimes \Theta^\vee} & M^\vee \ar@{>}[d]^{1\otimes \Theta^\vee}  \\
F_R^e(N^\vee) \ar@{>}[r]^{F_R^e (\rho^\vee)} & F_R^e(M^\vee) \\
}
\end{equation*}
Define $\mathcal{C}_e$ to be the category of
Artinian $R[\Theta; f^e]$-modules and $\mathcal{D}_e$ the category of $R$-linear maps $N \rightarrow F^e_R (N)$ for Noetherian $R$-modules $N$,
where morphisms in $\mathcal{D}_e$ are commutative diagrams
\begin{equation}\label{CD1}
\xymatrix{
N \ar@{>}[r]^{\varphi} \ar@{>}[d]^{\xi} & M \ar@{>}[d]^{\zeta}  \\
F_R^e(N) \ar@{>}[r]^{F_R^e(\varphi)} & F_R^e(M) \\
}.
\end{equation}
The construction above yields a contravariant functor $\Delta^e : \mathcal{C}_e \rightarrow \mathcal{D}_e$, and this functor is exact \cite[Chapter 10]{BrodmannSharpLocalCohomology}.
(Readers familiar with G.~Lyubeznik's notion of $F$-finite $F$-modules might recognize $\Delta^1$ as the first step in the direct limit system
whose direct limit yields Lyubeznik's functor $\mathcal{H}$; see section 4 in \cite{LyubeznikFModulesApplicationsToLocalCohomology}).
Furthermore, an application of the Matlis dual to (\ref{CD1}) yields
\begin{equation}\label{CD2}
\xymatrix{
F_R^e(M^\vee) \ar@{>}[r]^{F_R^e(\varphi^\vee)} \ar@{>}[d]^{\zeta^\vee} & F_R^e(N^\vee) \ar@{>}[d]^{\xi^\vee}\\
M^\vee \ar@{>}[r]^{\varphi^\vee}  & N^\vee   \\
}.
\end{equation}
which can be used to equip $M^\vee$ and $N^\vee$ with $R[\Theta; f^e]$-module structures given by
$\Theta m=\zeta^\vee (1\otimes m)$ and $\Theta n=\xi^\vee (1\otimes n)$, respectively.
With these structures, $\varphi^\vee$ is $R[\Theta; f^e]$-linear.
This construction yields an exact contravariant functor $\Psi^e : \mathcal{D}_e \rightarrow \mathcal{C}_e$. Now, after the identification of the double Matlis dual $(-)^{\vee\vee}$ with the identity
functor on Artinian and Noetherian $R$-modules, the compositions $\Psi^e \circ \Delta^e$ and $\Delta^e \circ \Psi^e$ yield the identity functors on
$\mathcal{C}_e$ and $\mathcal{D}_e$, respectively. In this sense, we can think of $\Delta^e$ as the Matlis dual that keeps track of a given Frobenius map.
An immediate corollary of this is the following.

\begin{cor}
\begin{enumerate}
\item[(a)] Let $M\in \mathcal{C}^e$. The action of $\Theta$ on $M$ is zero if and only if the map $\Delta^e(M)$ is zero.
\item[(b)] The map $\left( N\xrightarrow{\varphi} F_R^e(N) \right)\in \mathcal{D}^e$ is zero if and only if the action of
$\Theta$ on $\Psi^e\left( N\xrightarrow{\varphi} F_R^e(N) \right)$ is zero.
\end{enumerate}
\end{cor}

\section{Natural Frobenius maps on local cohomology modules}\label{Section: Natural Frobenius maps on local cohomology modules}

For any commutative ring $S$ of prime characteristic $p$
the local cohomology modules $\HH^\bullet_\mathfrak{m}(S)$ come equipped with a natural Frobenius map described as follows.
We have an $R$-linear map $g:S \rightarrow F_*^e S$ given by $g(s)=F_*^e s^{p^e}$ and this induces $R$-linear maps
$h: \HH^\bullet_\mathfrak{m}(S) \rightarrow \HH^\bullet_\mathfrak{m}(F_*^e S)$.
Now the Independence Theorem for local cohomology \cite[4.1]{BrodmannSharpLocalCohomology} gives
$$\HH^\bullet_\mathfrak{m}(F_*^e S)=\HH^\bullet_{\mathfrak{m} F_*^e R}(F_*^e S)= \HH^\bullet_{F_*^e \mathfrak{ \mathfrak{m}}^{[p]}}(F_*^e S)= \HH^\bullet_{F_*^e \mathfrak{m}}(F_*^e S)$$
and this, as an  $F_*^e S$-module, can be identified with $F_*^e \HH^\bullet_\mathfrak{m}(S)$.

Our next aim is to give a more explicit description of these Frobenius maps with the aid of Local Duality over $R$ \cite[11.2.5]{BrodmannSharpLocalCohomology}.

We adopt henceforth in this section the notation as in \ref{Notation: complete regular local}.

\begin{prop}\label{Proposition: Phi of local cohomology}(cf.~\cite[section 2]{LyubeznikVanishingLCCharp})
Consider $\HH^i_\mathfrak{m}(S)$ as an $R[T; f^e]$ module where $T$ acts as the
natural Frobenius map.
Then $\Delta^e( \HH^i_\mathfrak{m}(S) )$ is isomorphic to the map $\Ext_R^{d-i}(R/I, R) \rightarrow \Ext_R^{d-i}(R/I^{[p^e]}, R)$ induced by the surjection
$R/I^{[p^e]} \rightarrow R/I$.
\end{prop}

\begin{proof}
Let $h:  \HH^i_\mathfrak{m}(S) \rightarrow F_*^e \HH^i_\mathfrak{m}(S)$ be the $R$-linear map
described above. We apply Local Duality over the regular ring $R$ to
identify the functors $\HH^i_\mathfrak{m}(-)$ and $\Ext_R^{d-i}(-, R)^\vee$
and hence also $h$ with the map (which we also call $h$)
$h:  \Ext_R^{d-i}(S, R)^\vee \rightarrow F_*^e \Ext_R^{d-i}(S, R)^\vee$.
This map is also induced by the $R$-linear map $g$ above.

Now the map $1\otimes \widetilde{h}: F_R^e (\HH^i_\mathfrak{m}(S)) \rightarrow \HH^i_\mathfrak{m}(S)$
is identified with
$1\otimes \widetilde{h}: F_R^e (\Ext_R^{d-i}(S, R)^\vee) \rightarrow \Ext_R^{d-i}(S, R)^\vee$
and the natural identification
$F_R^e (\Ext_R^{d-i}(S, R)^\vee)=F_R^e (\Ext_R^{d-i}(S, R))^\vee$ yields the map
$1\otimes \widetilde{h}: F_R^e (\Ext_R^{d-i}(S, R))^\vee \rightarrow \Ext_R^{d-i}(S, R)^\vee$
which is the dual of the map
$w: \Ext_R^{d-i}(S, R) \rightarrow F_R^e (\Ext_R^{d-i}(S, R)) = \Ext_R^{d-i}(F_R^e (S), R)$
induced by
$1\otimes \widetilde{g} : F_R^e(S) \rightarrow S$ and where the last equality
follows from the flatness of the functor $F^e_R$ as $R$ is regular (cf.~\cite[Corollary 2.7]{KunzCharacterizationsOfRegularLocalRings}).
\end{proof}

Proposition \ref{Proposition: Phi of local cohomology} is the key ingredient that will allow us to answer questions
about local cohomology modules and their natural Frobenius actions in terms of $R$-linear maps of Noetherian modules into their Frobenius functors.
To do so we expand our prime characteristic toolbox in the next section.

\section{Further prime characteristic tools: roots and $\star$-closures}\label{Section: Further prime characteristic tools}

In this section we introduce two crucial constructions:
the first is a generalization of the $p^e$-root operation on ideals introduced in
\cite{BlickleMustataSmithDiscretenessAndRationalityOfFThresholds} (denoted there as $(-)^{[1/p^e]}$) and in
\cite{KatzmanParameterTestIdealOfCMRings}.

\begin{defi}
Let $e\geq 0$. Let $T$ be a commutative ring.
\begin{enumerate}
  \item[(a)] Given any matrix (or vector) $A$ with entries in $T$, we define $A^{[p^e]}$ to be the matrix obtained from $A$ by raising its
entries to the $p^e$th power.

  \item[(b)] Given any submodule $K$ of a finitely generated free module $T^\alpha$, we define $K^{[p^e]}$ to be the $R$-submodule of $T^\alpha$ generated by
$\{ v^{[p^e]} \,|\, v\in K \}$.

\end{enumerate}
\end{defi}

Henceforth in this section, $T$ will denote a regular ring with the property that $F_*^e T$ are \emph{intersection flat $T$-modules} for all $e\geq 0$, i.e.,
for any family of $T$-modules $\{M_\lambda\}_{\lambda\in\Lambda}$,
$$F_*^e T \otimes_T \bigcap_{\lambda\in\Lambda} M_\lambda= \bigcap_{\lambda\in\Lambda}\left( F_*^e T \otimes_T  M_\lambda\right) .$$
These include rings $T$ for which $F_*^e T$ are free $T$-modules (e.g.~, polynomial rings and power series rings with $F$-finite coefficient fields),
and also all complete regular rings (cf.~\cite[Proposition 5.3]{KatzmanParameterTestIdealOfCMRings}).
These rings have the property that for any collection of submodules $\{L_\lambda\}_{\lambda\in \Lambda}$ of $T^\alpha$ ,
$\left( \bigcap_{\lambda \in \Lambda} L_\lambda \right)^{[p^e]} =\bigcap_{\lambda \in \Lambda} L_\lambda^{[p^e]}$:
indeed, $T$ being regular implies that for any submodule $L\subseteq T^\alpha$, $L^{[p^e]}$ can be identified with $F_T^e( L)$ 
and of $F_*^e T$ being intersection-flat implies
$$F_T^e ( \bigcap_{\lambda \in \Lambda} L_\lambda )= F_*^e T \otimes_T  \bigcap_{\lambda \in \Lambda} L_\lambda =
\bigcap_{\lambda \in \Lambda} F_*^e T \otimes_T  L_\lambda = \bigcap_{\lambda \in \Lambda} F_T^e ( L_\lambda ).
$$

The theorem below extends the $I_e(-)$ operation defined on ideals in \cite[Section 5]{KatzmanParameterTestIdealOfCMRings}
and in \cite[Definition 2.2]{BlickleMustataSmithDiscretenessAndRationalityOfFThresholds}
(where it is denoted $(-)^{[1/p^e]}$) to submodules of free $R$-modules.
Recall that for an ideal $J$ of a regular ring, $I_e(J)$ was defined as the smallest ideal whose $p^e$th Frobenius power contains $J$.
We extend this as follows.

\begin{thm}
\label{Theorem: qth root with respect to U}
Let $e\geq 1$.
\begin{enumerate}
  \item[(a)] Given a submodule $K\subseteq T^\alpha$ there exists a minimal submodule $L \subseteq T^\alpha$ for which
  $K\subseteq L^{[p^e]}$. We denote this minimal submodule $I_e (K)$.
  \item[(b)] Let  $U$ be an $\alpha\times \alpha$ matrix with entries in $T$ and let $V\subseteq T^\alpha$.
  The set of all submodules $W \subseteq T^\alpha$ which contain $V$ and which satisfy $U W \subseteq W^{[p^e]}$ has a unique
  minimal element.
\end{enumerate}
\end{thm}
\begin{proof}
See \cite[Theorem 3.2]{KatzmanZhangAnnihilatorsOfArtinianModules}.
%
%
\end{proof}

\begin{defi}
With notation as in Theorem \ref{Theorem: qth root with respect to U}, we call the unique minimal submodule in
\ref{Theorem: qth root with respect to U}(b) the \emph{star-closure of $V$ with respect to $U$}
and denote it
$V^{\star U}$.
\end{defi}

The effective calculation of the $\star$-closure boils down to the calculation of $I_e$. When $F_*^e T$ is $T$-free,  this is
a straightforward generalization of the calculation of $I_e$ for ideals. To do so,  fix a free basis $\mathcal{B}$ for $F_*^e T$ and note that
every element $v\in T^\alpha$ can be expressed uniquely in the form $v=\sum_{b\in \mathcal{B}} u_{b}^{[p^e]} b$
where $u_{b}\in T^\alpha$ for all $b\in \mathcal{B}$.

\begin{prop}
\label{Proposition: Computing Ie}
Let $e\geq 1$.
\begin{enumerate}
  \item [(a)] For any submodules $V_1, \dots, V_\ell\subseteq R^n$, $I_e(V_1 + \dots + V_\ell)=I_e(V_1)  + \dots + I_e(V_\ell)$.
  \item [(b)] Let $\mathcal{B}$ be a  free basis for $F_*^e T$.
  Let $v\in R^\alpha$ and let
  $$v=\sum_{b\in \mathcal{B}} u_{b}^{[p^e]} b $$
  be the unique expression for $v$ where $u_{b}\in T^\alpha$ for all
  $b\in \mathcal{B}$. Then  $I_e(T v)$ is the submodule $W$ of $T^\alpha$ generated by $\{ u_b  \,|\, b\in \mathcal{B} \}$.
  \end{enumerate}
\end{prop}

\begin{proof}
See \cite[Proposition 3.4]{KatzmanZhangAnnihilatorsOfArtinianModules}.
%
%
\end{proof}

The behavior of the $I_e$ operation and the $\star$-closure under localization and completion will be crucial for obtaining the main result of this paper.
To investigate the latter we need the following generalization of \cite[Lemma 6.6]{LyubeznikSmithCommutationOfTestIdealWithLocalization}.

\begin{lem}\label{Lemma: commutation with localization and completion}
Let $\mathcal{T}$ be a completion of $T$ at a prime ideal $P$.
Let $\alpha\geq 0$ and let $W$ be a submodule of $\mathcal{T}^\alpha$.
For all $e\geq 0$, $W^{[p^e]} \cap T^\alpha=(W\cap T^\alpha)^{[p^e]}$.
\end{lem}

\begin{proof}
If $T$ is local with maximal ideal $P$, the result follows from a straightforward modification of the proof
of \cite[Lemma 6.6]{LyubeznikSmithCommutationOfTestIdealWithLocalization}.

We now reduce the general case to the previous case which implies that
$W^{[p^e]} \cap T_P^\alpha=(W\cap T_P^\alpha)^{[p^e]}$. Intersecting with $T^\alpha$ now gives\\
$\displaystyle W^{[p^e]} \cap T^\alpha=(W\cap T_P^\alpha)^{[p^e]}\cap T^\alpha=(W\cap T_P^\alpha\cap T^\alpha)^{[p^e]}=(W\cap T^\alpha)^{[p^e]}$.
\end{proof}

\begin{lem}[cf.~Proposition 7 in \cite{Murru}]\label{Lemma: Ie and localization}
Let $\mathcal{T}$ be a localization of $T$ or a completion at a prime ideal.
For all $e\geq 1$, and all submodules $V\subseteq T^\alpha$, $I_e(V \otimes_T \mathcal{T})$ exists and equals $I_e(V)\otimes_T \mathcal{T}$.
\end{lem}

\begin{proof}
Let $L\subseteq \mathcal{T}^\alpha$ be a submodule, such that
$L^{[p^e]} \supseteq V \otimes_T \mathcal{T}$.
We clearly have
$L^{[p^e]} \cap T^\alpha= (L \cap T^\alpha)^{[p^e]}$ when
$\mathcal{T}$ is a localization of $T$ and when
$\mathcal{T}$ is a completion of $T$ this follows from
the previous Lemma.
We deduce that $(L \cap T^\alpha)\supseteq I_e (V \otimes_T \mathcal{T} \cap T^\alpha)$
and hence
$L\supseteq  (L \cap T^\alpha) \otimes_T \mathcal{T} \supseteq I_e (V \otimes_T \mathcal{T} \cap T^\alpha) \otimes_T \mathcal{T}$.

But since $I_e (V \otimes_T \mathcal{T} \cap T^\alpha) \otimes_T \mathcal{T}$ satisfies
$$\left( I_e (V \otimes_T \mathcal{T} \cap T^\alpha) \otimes_T \mathcal{T}\right) ^{[p^e]}
=
I_e (V \otimes_T \mathcal{T} \cap T^\alpha)^{[p^e]} \otimes_T \mathcal{T}
\supseteq
(V \otimes_T \mathcal{T} \cap T^\alpha) \otimes_T \mathcal{T}
\supseteq
V \otimes_T \mathcal{T}$$
we deduce that  $I_e (V \otimes_T \mathcal{T} \cap T^\alpha) \otimes_T \mathcal{T}$
is the smallest submodule $K\subseteq \mathcal{T}^\alpha$ for which
$K^{[p^e]}\supseteq V \otimes_T \mathcal{T}$.
We conclude that $I_e(V \otimes_T \mathcal{T})$  equals
$I_e (V \otimes_T \mathcal{T} \cap T^\alpha) \otimes_T \mathcal{T}$.


We always have
$$I_e(V \otimes_T \mathcal{T}) = I_e (V \otimes_T \mathcal{T} \cap T^\alpha) \otimes_T \mathcal{T} \supseteq I_e(V)\otimes_T \mathcal{T}.$$
On the other hand
$$ \left( I_e(V)\otimes_T \mathcal{T} \right)^{[p^e]} =
I_e(V)^{[p^e]}\otimes_T \mathcal{T} \supseteq
V \otimes_T \mathcal{T} $$
hence
$I_e(V \otimes_T \mathcal{T})  \subseteq I_e(V)\otimes_T \mathcal{T}$
and thus
$I_e(V \otimes_T \mathcal{T})  = I_e(V)\otimes_T \mathcal{T}$.
\end{proof}

\begin{lem}\label{Lemma: star commutes with completion}
Let $e\geq 1$, let  $U$ be a $\alpha\times \alpha$ matrix with entries in $T$ and let $V\subseteq T^\alpha$.
For any prime $P\subset T$,
$$\widehat{V_P}^{\star U}=\widehat{V^{\star U}}_P .$$
\end{lem}

\begin{proof}
As in the proof of \cite[Theorem 3.2]{KatzmanZhangAnnihilatorsOfArtinianModules}
define inductively $V_0=V$ and
$V_{i+1}=I_1(U V_i) + V_i$ for all $i\geq 0$, and also
$W_0=\widehat{V_P}$ and
$W_{i+1}=I_1(U W_i) + W_i$ and note that
the ascending chains of modules $\{V_i\}_{\{i\geq 0\}}$ and $\{W_i\}_{\{i\geq 0\}}$ have stable values $V^{\star U}$ and $\widehat{V}^{\star U}$, respectively.
An easy induction on $i\geq 0$ together with Lemma \ref{Lemma: Ie and localization}  shows that $W_{i}=\widehat{(V_{i})_P}$ for all $i\geq 0$, and the result follows.
\end{proof}

\section{Another look at $\mathcal{C}_e$}
Throughout this section we adopt the notation of \ref{Notation: complete regular local}
and  have a closer look at Artinian $R[\Theta; f^e]$-modules, starting with the possible
$R[\Theta; f^e]$-modules structures of $E^\alpha$ for $\alpha\geq 1$.

The module $\mathcal{F}^e(E)$ of Frobenius maps on $E$ is isomorphic to $R$ (\cite[Example 3.7]{LyubeznikSmithCommutationOfTestIdealWithLocalization})
and the generator of this module, the \emph{natural Frobenius map on E}, can be described explicitly as follows.
The regular local ring $R$ is isomorphic to a power series ring $\mathbb{K}[\![ x_1, \dots, x_n ]\!]$ for some field $\mathbb{K}$ of
characteristic $p$, and $E$ is then isomorphic to the module of inverse polynomials
$\mathbb{K}[ x_1^-, \dots, x_n^- ]$ (cf.~Example 12.4.1 in \cite{BrodmannSharpLocalCohomology})
which has a Frobenius map given by
$T \lambda x_1^{\alpha_1} \cdot \ldots \cdot x_n^{\alpha_n}=\lambda^p x_1^{p\alpha_1} \cdot \ldots \cdot x_n^{p\alpha_n}$ for all $\alpha_1, \dots, \alpha_n<0$.
Now $\Theta = T^e$ is a generator for $\mathcal{F}^e(E)$, and
$\mathcal{F}^e(E^\alpha)$ can be identified with the $R$-module of $\alpha\times \alpha$ matrices with entries  in $R$: we associate to such a matrix $U$
the Frobenius map $\Theta: E^\alpha \rightarrow E^\alpha$ given by
$$\Theta
\left[
\begin{array}{l}
z_1\\
 \vdots\\
z_\alpha
\end{array}
\right]=
U^t
\left[
\begin{array}{l}
T^e z_1\\
 \vdots\\
T^e z_\alpha
\end{array}
\right]
.$$

Now any Artinian $R[\Theta; f^e]$-module $M$ can be embedded into $E^\alpha$ for some $\alpha\geq 1$ and
an application of the Matlis dual yields a diagram with exact rows
\begin{equation}\label{CD3}
\xymatrix{
R^\alpha \ar@{>}[r]^{h}  & M^\vee \ar@{>}[d]^{(1\otimes \Theta)^\vee}   \ar@{>}[r]^{} & 0\\
R^\alpha \ar@{>}[r]_{F_R^e(h)}   & F_R^e(M^\vee)   \ar@{>}[r]^{} &0 \\
}.
\end{equation}
where the vertical map is $\Delta^e(M)$.
Choose a presentation $\Image R^\beta \xrightarrow{A} R^\alpha$ of $\ker h$.
If we identify $M^\vee$ with $\Coker A$,
we can then identify the vertical map above with multiplication by some $\alpha \times \alpha$ matrix $U$ (which must satisfy $\Image U A \subseteq \Image A^{[p^e]}$.)
An application of $\Psi^e$ to the commutative diagram
\begin{equation}\label{CD4}
\xymatrix{
R^\alpha \ar@{>}[r]^{h}  \ar@{>}[d]^{U} & \Coker A \ar@{>}[d]^{U}   \ar@{>}[r]^{} & 0\\
R^\alpha \ar@{>}[r]_{F_R^e(h)}   & \Coker A^{[p^e]}  \ar@{>}[r]^{} &0 \\
}.
\end{equation}
gives an embedding $\Ann_{E^\alpha} A^t \subseteq E^\alpha$ of $R[\Theta; f^e]$-modules where
$\Ann_{E^\alpha} A^t$ is isomorphic to $M$ as $R[\Theta; f^e]$-modules and where the action of $\Theta$ on $E^\alpha$ is given by $U^t T$.

We summarize this discussion with the following corollary.
\begin{cor}
Any Artinian $R[\Theta; f^e]$-module can be embedded as an $R[\Theta; f^e]$-module into some $E^\alpha$ on which the action of $\Theta$ is given by $U^t T$.
Consequently, any such module is isomorphic to $\Ann_{E^\alpha} A^t$ for some $\beta \times \alpha$ matrix $A$ and where
the action of $\Theta$ is the restriction of the action of $U^t T^e$ of $E^\alpha$ to $\Ann_{E^\alpha} A^t$ where $U$ is a
$\alpha \times \alpha$ matrix satisfying $\Image U A \subseteq \Image A^{[p^e]}$.
\end{cor}

\section{The global parameter test ideal}
\label{Section: The global parameter test ideal}

As the first application of the tools we have developed thus far, in this section we will construct a global parameter test ideal, {\it i.e.,}
an ideal whose localization at each prime ideal is the parameter test ideal of the localization of the ring.
To this end, we will use $R$ to denote a regular Noetherian ring of prime characteristic $p$ such that $F^e_*R$ are intersection flat for all $e\geq 0$. We will use $S$ to denote a
homomorphic image of $R$
which has a completely stable parameter test element (cf.~\cite[Section 6]{HochsterHunekeTC1}).
If $S$ has a parameter test element, then the parameter test ideal $\tau$
is generated by all parameter test elements, the non-$F$-rational locus of $S$ is given by $V(\tau)$, and we recover the
well known fact that the $F$-rational locus is open (cf.~\cite{VelezOpennessOfTheFRationalLocus}).

Given any \emph{local} ring $(S,\mathfrak{m})$ we denote
\[
\tau^S=\bigcap_{{\rm parameter\ ideals\ }I\ {\rm in\ }S}(I:I^*)
\]
the {\it parameter test ideal of $S$}.
In \cite{SmithTestIdeals} Karen Smith showed that
when $S$ is local and Cohen-Macaulay, $\tau^S$ localizes to parameter test ideals of the localization.
The results of this section extend the existence of such ideals to the non-local case
(Theorem \ref{Theorem: global pti in CM rings}), to the generalized Cohen-Macaulay case
(Theorem \ref{theorem:  pti for genCM}), and to rings with finitely many isolated non-Cohen Macaulay point,
each being an isolated non-$F$ rational points
(Theorem \ref{Theorem: pti in isolated non-CM points}).
Furthermore, these ideals $\tau$ are given explicitly in terms of operations on ideals
which are readily computable.\footnote{See Macaulay2 package ``FSing'' \cite{Macaulay}.}

First we note that rings with completely stable test elements are abundant.

\begin{thm}
\label{Theorem: existence of test elements}
The following rings possess completely stable test elements:
\begin{itemize}
\item Geometrically reduced, finitely generated algebras over any field of characteristic $p$ (\cite[Corollary 1.5.5]{HochsterHunekeTightClosureInEqualCharactersticZero}).
In this setting completely stable test elements are easily computable as minors of a Jacobian matrix.
\item $F$-finite and reduced rings (cf.~\cite[Theorem 5.10]{HochsterHunekeFRegularityTestElementsBaseChange}).
\item Reduced, finitely generated $B$-algebras where $(B, m, K)$ is a complete local ring with coefficient field $K$ of characteristic $p$
(cf.~\cite[Theorem 6.20]{HochsterHunekeFRegularityTestElementsBaseChange}).
\end{itemize}
\end{thm}

The result of this section rely of the following result.

\begin{thm}
[cf.~Proposition 2.5 in \cite{SmithFRatImpliesRat}]
\label{Theorem: characterization of parameter test ideal}
Let $(S, \mathfrak{m})$ be a local Cohen-Macaulay ring with parameter test-element $c$.
Write $H=\HH^{\dim S}_\mathfrak{m}(S)$, $G=\{ a\in H \,|\, c a=0 \}$ and let $L$ be the largest
$S$-submodule of $G$ stable under the natural Frobenius map on $H$.
The parameter test ideal $\tau$ of $S$ is
the annihilator of $\displaystyle 0^*_{{\HH^{\dim S}_{\fm}} (S)}$ and is
given by $(0:_S L)$.
\end{thm}
\begin{proof}
The first statement is a restatement of \cite[Proposition 4.4(ii)]{SmithTestIdeals}.
Fix any system of parameters  $x_1, \dots, x_d\in \mathfrak{m}$;
$H$ can be constructed as a direct limit of Koszul cohomology, resulting in
\begin{equation}\label{Equation: 2}
H=\lim_{\rightarrow\atop{t}} S/(x_1^t, \dots, x_d^t)
\end{equation}
where the maps in the direct limit system are given by multiplication by $x=x_1 \dots x_d$.
If $S$ is Cohen-Macaulay, as in the hypothesis of this theorem, $x_1, \dots, x_d$ form a regular sequence and the maps in the direct limit system
(\ref{Equation: 2}) are injective.

Now every element in $H$ is the image of some $s + (x_1^t, \dots, x_d^t) \in S/(x_1^t, \dots, x_d^t)$ in the direct limit, and
we denote such an element $[s + (x_1^t, \dots, x_d^t)]$.
We can also describe the action of natural Frobenius $T$ on $H$: $[s + (x_1^t, \dots, x_d^t)]$
is mapped to
$[s^p + (x_1^{pt}, \dots, x_d^{pt})]$.

Let $t\geq 1$ and pick any $s\in S$.
Since $c$ is a test-element,
$c s^{p^e} \in (x_1^{p^e t}, \dots, x_d^{p^e t})$ for all $e\geq 0$
if and only if
$s\in (x_1^t, \dots, x_d^t)^*$,
hence
$$L=\bigcap_{e\geq 0} \Ann_H c T^e=
\lim_{\rightarrow\atop{t}} (x_1^t, \dots, x_d^t)^*/(x_1^t, \dots, x_d^t) .$$

Clearly, $\tau L=0$ and $\tau \subseteq (0:_S L)$.
Also $(0:_S L) L=0$  and the injectivity of the maps in (\ref{Equation: 2})
implies $(0:_S L) \left((x_1^t, \dots, x_d^t)^{[p^e]}\right)^* \subseteq
(x_1^{p^e t}, \dots, x_d^{p^e t})$ for all $e\geq 0$, hence $(0:_S L) \subseteq \tau$.
\end{proof}

Write $S=R/I$ where $R$ is a regular Noetherian ring of prime characteristic $p$ such that $F^e_*R$ are intersection flat for all $e\geq 0$. We
represent the map
$\theta: \Ext^{\dim(R) -\dim(S)}_{R} (R/I, R)
\xrightarrow[]{} \Ext^{\dim(R) -\dim(S)}_{R} (R/I^{[p]},R) $ as
$\Coker A \xrightarrow[]{U} \Coker A^{[p]}$ for
a $\alpha\times \beta$ matrix $A$ and $\alpha \times \alpha$ matrix $U$ with entries in $R$.

We now obtain the following non-local result.

\begin{thm}\label{Theorem: global pti in CM rings}
Let  $c\in R$  be such that its images in all completions of $S$ at all prime ideals are parameter test-elements.
The localization (resp.~completion) of
$$\left(0 :_{R} \frac{R^\alpha}{(\Image A+ c R^\alpha)^{\star U}} \right)$$
at any prime ideal $P$ in the Cohen-Macaulay locus of $S$
is a parameter test-ideal of $S_P$ ($\widehat{S_P}$, resp.).
\end{thm}

\begin{proof}

Let $P\subset R$ be a prime ideal in the Cohen-Macaulay locus of $S$.
Since  $\HH_{P {R_P}}^{\dim  {S_P}} \left({S_P}\right)$ is a
$\widehat{S_P}$-module and is isomorphic to
$\HH_{P \widehat{R_P}}^{\dim  \widehat{S_P}} \left(\widehat{S_P}\right)$
as $\widehat{S_P}$-modules, Theorem \ref{Theorem: characterization of parameter test ideal} implies that it is enough
to show that the completion at $P$ of
$$\left(0 :_{R} \frac{R^\alpha}{(\Image A+ c R^\alpha)^{\star U}} \right)$$
is a parameter test-ideal of $\widehat{S_P}$.

With the notation as in section \ref{Section: Prime characteristic tools}
$$\Delta^1\left( \HH_{P \widehat{R_P}}^{\dim  \widehat{S_P}} \left(S_P\right)\right)$$
is isomorphic to $\Coker \widehat{A} \xrightarrow[]{\widehat{U}} \Coker \widehat{A}^{[p]}$,
the completion at $P$ of $\Coker A \xrightarrow[]{U} \Coker A^{[p]}$.

Let $V\subseteq \HH_{P \widehat{R_P}}^{\dim  \widehat{S_P}} \left(S_P\right)$ be any $\widehat{R_P}[\Theta; f]$-submodule killed by
$c$. An application of $\Delta^1\left(-\right)$ to this inclusion of $\widehat{R_P}[\Theta; f]$-modules yields a commutative diagram with exact rows
\begin{equation}\label{CD3}
\xymatrix{
\Coker \widehat{A} \ar@{>}[r] \ar@{>}[d]^{U} & \Coker B \ar@{>}[r] \ar@{>}[d]^{U} &  0 \\
\Coker \widehat{A}^{[p]} \ar@{>}[r]& \Coker B^{[p]}\ar@{>}[r] & 0  \\
}.
\end{equation}
where $B$ is a $\alpha \times \gamma$ matrix with entries in $\widehat{R_P}$ with the following three properties
\begin{enumerate}
\item[(a)] $\Image B \supseteq \Image \widehat{A}$,
\item[(b)] $\Image U B \subseteq \Image B^{[p]}$, and
\item[(c)] $c R_P^\alpha \subseteq \Image B$.
\end{enumerate}
If $V=\left(\Coker B\right)^\vee$ is the \emph{largest} $\widehat{R_P}[\Theta; f]$-submodule of
$\HH_{P \widehat{R_P}}^{\dim  \widehat{S_P}} \left(S_P\right)$
killed by $c$, then the matrix $B$ has the \emph{smallest} image among those matrices satifying (a), (b) and (c) above.
Such a matrix is given by one whose image is $\left( \Image \widehat{A} + c \widehat{R_P}^\alpha \right)^{\star U}$.
We apply Theorem \ref{Theorem: characterization of parameter test ideal} and
deduce that the parameter test ideal at the completion is given by
$$\left(0 :_{\widehat{R_P}} \frac{\widehat{R_P}^\alpha}{(\widehat{A}+ c \widehat{R_P}^\alpha)^{\star U}} \right) .$$

We established in Lemma \ref{Lemma: star commutes with completion} that
the ${(-)}^\star$ operation commutes with localization and completion, and the additional use of the flatness of completion
implies that  parameter test ideal of  $S_P$ (resp.~$\widehat{S_P}$)
is given by the localization (resp.~completion) of
$$\left(0 :_{R} \frac{R^\alpha}{(A+ c R^\alpha)^{\star U}} \right)$$
at $P$.
\end{proof}

The rest of this section extends Theorem \ref{Theorem: global pti in CM rings} to a wider class of rings, and we start with
the following.

\begin{defi}
A local ring $(S,\mathfrak{m})$ is called a \emph{generalized Cohen-Macaulay} if $\HH^i_\mathfrak{m}(S)$ has finite length for all $i<\dim S$.
\end{defi}

Given any local ring $(S,\mathfrak{m})$ we denote
\[
\tau^S=\bigcap_{{\rm parameter\ ideals\ }I}(I:I^*)
\]
the {\it parameter test ideal of $S$};
\cite[Proposition 2.5]{SmithFRatImpliesRat} shows that
$\tau^S\subseteq (0 :_S 0^*_{H^{\dim S}_{\mathfrak{m}}(S)})$.

\begin{defi}
\label{definition: finding colon-killers}
Let $(S,\fm)$ be a Noetherian local ring. An element $c\in S$ is called a colon-killer if
$c((x_1,\dots,x_i):x_{i+1})\subseteq (x_1,\dots,x_i)$ for all parameters $x_1,\dots,x_{i+1}$ in $S$.
It is straightforward to check that the set of colon-killers forms an ideal which we will denote with $\fc$.
\end{defi}

\begin{lem}
\label{lemma: J and tau}
Let $(S,\fm)$ be a Noetherian commutative local ring of characteristic $p$ and let $\fc$ be as above. We have
\[\fc^d (0 :_S 0^*_{H^d_{\fm}(S)})\subseteq \tau^S.\]
\end{lem}
\begin{proof}
Let $c_1,\dots,c_d$ be elements of $\fc$. It suffices to show that $c_1\cdots c_d\Ann_S(0^*_{H^d_{\fm}(S)})\subseteq \tau^S$.
Let $x_1,\dots,x_d$ be a system of parameters and $z$ be an element of $S$.
Then according to \cite[Proposition 3.3]{SmithTightClosureParameter} one has $z\in (x_1,\dots,x_d)^*$ if and only if
$[\frac{z}{x_1\cdots x_d}]\in 0^*_{H^d_{\fm}(S)}$.
Assume that $\delta\in \Ann_S (0^*_{H^d_{\fm}(S)})$ and
hence $\delta[\frac{z}{x_1\cdots x_d}]=0$.
This is equivalent to the existence of an integer $n$ such that
\[\delta z\in (x_1^{n+1},\dots,x^{n+1}_d):(x_1\cdots x_d)^n\]
Write $\delta z(x_1\cdots x_d)^n=\sum_ia_ix^{n+1}_i$;
we have $x^n_1(\delta z(x_2\dots x_d)^n-a_1x_1)\in (x^{n+1}_2,\dots,x^{n+1}_d)$
and  $c_1(\delta z(x_2\dots x_d)^n-a_1x_1)\in (x^{n+1}_2,\dots,x^{n+1}_d)$,
since $c_1$ is a colon-killer.
Write $c_1\delta z(x_2\dots x_d)^n=ca_1x_1+\sum_ib_ix^{n+1}_i$.
Then $x^n_2(c_1\delta z(x_3\cdots x_d)^n-b_2x_2) \in (x_1,x^{n+1}_3,\dots, x^{n+1}_d)$.
So,
\[c_2\Big(c_1\delta z(x_3\cdots x_d)^n-b_2x_2\Big) \in (x_1,x^{n+1}_3,\dots, x^{n+1}_d),
\]
i.~e.,
\[
c_1c_2\delta z (x_3\cdots x_d)^n \in (x_1,x_2,x^{n+1}_3,\dots, x^{n+1}_d).
\]
Continuing this process, we obtain $c_1\cdots c_d\delta z\in (x_1,\dots,x_d)$, i.~e.,
$c_1\cdots c_d\delta \in (x_1,\dots,x_d):(x_1,\dots,x_d)^*$.
\end{proof}

\begin{lem}
\label{lemma: ann of lc localizes}
Let $(S,\fm)$ be a Noetherian commutative local ring of characteristic $p$ that satisfies Serre's condition $S_2$. Then
\[\Ann_{R_P}(0^*_{H^h_{PR_P}(R_P)})\subseteq \Ann_R(0^*_{H^d_{\fm}(R)})_P\]
for each prime ideal $P$ of $R$, where $h=\height(P)$.
\end{lem}
\begin{proof}
This is essentially proved on page 3468 in \cite{SmithTestIdeals};
The key ingredient of the proof there is \cite[Lemma 2.1]{SmithTestIdeals}
which holds in any Noetherian ring satisfying the $S_2$ condition.
\end{proof}

\begin{thm}
\label{theorem:  pti for genCM}
Let $(S,\fm)$ be a $d$-dimensional generalized Cohen-Macaulay local ring, and write $\tau=\tau^S$.
Then $\tau_P$ is the parameter test ideal of $S_P$ for each prime ideal $P$ of $S$.
\end{thm}
\begin{proof}
If $S$ is Cohen-Macaulay, the result follows from Theorem \ref{Theorem: global pti in CM rings}
so we assume that $S$ is not Cohen-Macaulay.

First, we prove that $\tau_P\subseteq \tau^{S_P}$.

Let $x$ be an element of $\tau$. We may assume that $P\neq \fm$; by definition of $\tau$, it is the parameter test ideal of $S$.
Set $h=\height(P)$.
Let $x_1,\dots,x_d$ be a full system of parameters such that $\frac{x_1}{1},\dots,\frac{x_h}{1}$ is a system of parameters for $S_P$.
Since $S$ is generalized Cohen-Macaulay, $\frac{x_1}{1},\dots,\frac{x_h}{1}$ is a regular sequence;
consequently it suffices to show that
\[\frac{x}{1}(\frac{x^t_1}{1},\dots,\frac{x^t_h}{1})^*\subseteq (\frac{x^t_1}{1},\dots,\frac{x^t_h}{1})\]
for all $t\geq 1$.
It is proved in \cite{AberbachHochsterHunekeFinitePhantomProjDim} that tight closure commutes with localization for ideals parameters.
Thus, if $\frac{z}{1}\in (\frac{x^t_1}{1},\dots,\frac{x^t_h}{1})^*$, then there is an element $u\in R\backslash P$ such that
$uz\in (x^t_1,\dots, x^t_d)^*$.
Then $xuc\in (x^t_1,\dots, x^t_d)$ and hence $\frac{x}{1}\frac{z}{1}\in (\frac{x^t_1}{1},\dots,\frac{x^t_h}{1})$.
This finishes the proof of $\tau_P\subseteq \tau^{S_P}$.

To prove the reverse inclusion, we write
$J=\prod_{i=0}^{\dim S-1}\Ann(H^i_{\fm}(S))$ and apply
\cite[Corollary 8.1.3]{BrunsHerzog}, to deduce that $J\subseteq \fc$,
and by Lemma \ref{lemma: J and tau}, we have
\[J^d\Ann_S(0^*_{H^d_{\fm}(S)})\subseteq \tau.\]
Therefore, we have $\tau_P= \Ann_S(0^*_{H^d_{\fm}(S)})_P$ for each prime ideal $P\neq \fm$, since $J^d$ is $\fm$-primary.
Now Lemma \ref{lemma: ann of lc localizes} implies that
\[
\tau^{S_P}=\Ann_{S_P}(0^*_{H^h_{PS_P}(S_P)})\subseteq \Ann_S(0^*_{H^d_{\fm}(S)})_P=\tau_P,
\]
where the first equality follows from our assumption that $S_P$ is Cohen-Macaulay.
\end{proof}

Next we  extend Theorem \ref{theorem:  pti for genCM} to non-local cases.

\begin{thm}\label{Theorem: pti in isolated non-CM points}
Let $c\in R$ be as in Theorem \ref{Theorem: existence of test elements}.
Write $h=\dim(R)-\dim(S)$ and
$Z=\Ann\Big(\frac{R^{\alpha}}{(A+cR^{\alpha})^{\star}}\Big)$ as in the previous paragraph.
Assume that $S$ has isolated non-Cohen-Macaulay points $\fm_1,\dots,\fm_t$,
and that each of these is an isolated non-$F$-rational point.
Write $\tau_i=\tau^{S_{\fm_i}} \cap S$ for $1\leq i\leq t$ and $\tau:=Z\cap \tau_1\cap \cdots \cap \tau_t$.
Then $\tau$ is the global parameter test ideal of $S$ in the sense that
$\tau_P$ is the parameter test ideal of $R_P$ for each prime ideal $P$.
\end{thm}
\begin{proof}
The fact that  $\fm_1,\dots,\fm_t$ are isolated non-$F$-rational points implies that
$\tau_i$ is $\fm_i$ primary for each $1\leq i\leq t$.
Now for any $\fm\notin\{\fm_1,\dots,\fm_t\}$,
$\tau_{\fm}=Z_{\fm}$ is the parameter-test-ideal of the Cohen-Macaulay localization $S_{\fm}$.
Let $\fm = \fm_i$ for some  $1\leq i\leq t$.

Note that
$$Z_{\fm}=\Ann_{S_{\fm}} 0^*_{{\HH^{\dim S}_{\fm}}(S)}$$
hence ${\tau_i}_{\fm} \subseteq Z_{\fm}$ and
$\tau_{\fm}=Z_{\fm} \cap {\tau_i}_{\fm} =  {\tau_i}_{\fm}$ is
the parameter-test-ideal of the generalized Cohen-Macaulay localization $S_{\fm}$.
\end{proof}

It turns out that the idea in this section, in particular the use of $\Delta^e$ (the Matlis dual that keeps track of Frobenius), can also be applied to study HSL numbers of local cohomology modules, which will occupy our last section.

\section{HSL numbers}
\label{Section: HSL numbers}
In this section we will apply the machinery we have developed to investigate HSL numbers of local cohomology modules, which are defined as follows.

Given any commutative ring $T$ of prime characteristic and $T[\Theta; f]$-module $M$, we declare $m\in M$ to be \emph{nilpotent} under the action of $\Theta$ if $\Theta^e m=0$ for some
$e\geq 0$. The set $\Nil(M)$ of nilpotent elements under $\theta$ forms a $T[\Theta; f]$-submodule of $M$.

\begin{thm}\label{Theorem: HSL}
(cf.~ \cite[Proposition 1.11]{HartshorneSpeiserLocalCohomologyInCharacteristicP}, \cite[Proposition 4.4]{LyubeznikFModulesApplicationsToLocalCohomology})
Let $T$ be a quotient of a complete regular ring, and let $M$ be an Artinian $T[\Theta; f]$-module.
There exists an $\eta\geq 0$ such that $\Theta^\eta \Nil(M)=0$.
\end{thm}

\begin{defi}
The HSL number of a $T[\Theta; f]$-module $M$ is $\inf \{ \eta\geq 0 \,|\, \Theta^\eta \Nil(M)=0 \}$.
\end{defi}
Thus, under the hypothesis of Theorem \ref{Theorem: HSL}, HSL numbers are finite.

We will first compute the HSL numbers of local cohomology modules of quotients of complete regular local rings, and later describe the loci of primes $P$ in quotients of
polynomial rings on which the HSL numbers of local cohomology modules of the localization at $P$ are bounded by a given integer.

Let $(R,m)$ be a complete regular ring and let $H$ be an $R[\Theta; f]$-module.
Write $\Delta^1(H)=(\Coker A \xrightarrow{U} \Coker A^{[p]})$ where $A$ is a $\alpha \times \beta$ matrix with entries in $R$, and $U$ is a $\alpha\times \alpha$ matrix with entries in $R$.
Note that $H$ is a $R[\Theta^e; f^e]$-module, and that
$\Delta^e(H)=(\Coker A \xrightarrow{ U^{[p^{e-1}]} \cdots U^{[p]} U} \Coker A^{[p]})$.

Define now the $R[\Theta^e, f^e]$ submodule $H_e=\{ h\in H \,|\, \Theta^e h  = 0 \}$.
An application of $\Delta^e$ to the inclusion of $R[\Theta^e, f^e]$-modules $H_e\subseteq H$ yields
a commutative diagram
\begin{equation}\label{CD5}
\xymatrix{
 \Coker A  \ar@{>}[r]^{}  \ar@{>}[d]^{U^{[p^{e-1}]} \cdots U^{[p]} U } & \Coker B \ar@{>}[d]^{U^{[p^{e-1}]} \cdots U^{[p]} U}   \ar@{>}[r]^{} & 0\\
 \Coker A^{[p^e]}  \ar@{>}[r]^{} & \Coker B^{[p^e]}  \ar@{>}[r]^{} &0 \\
}.
\end{equation}
for some $\alpha \times \gamma$ matrix $B$.
Since the action of $\Theta^e$ is zero on $H_e$, the rightmost vertical map must vanish too.
Furthermore, as
$\Delta^e(H)$ is the largest $R[\Theta^e, f^e]$-submodule of $H$ which is killed by $\Theta^e$,
$B$ is characterized as the matrix with smallest image containing the image of $A$ and making (\ref{CD5}) commute. In other words,
$\Image B=\Image A + I_e (\Image U^{[p^{e-1}]} \cdots U^{[p]} U)$.

Write now $B_e=\Image A + I_e (\Image U^{[p^{e-1}]} \cdots U^{[p]} U)$ and consider the descending chain $\{ B_e \}_{e\geq 0}$ of submodules of $R^\alpha$.
We can now conclude that the HSL number of $H$ is the smallest value of $e\geq 0$ for which $B_e/B_{e+1}=0$.

This calculation also shows that the HSL numbers obtained after localization at a prime $P$ are bounded by the HSL numbers obtained after localizing
at any maximal ideal containing $P$. Thus we obtain the following corollary.

\begin{cor}\label{Corollary: $F$-injectivity localizes}
$F$-injectivity localizes for complete local rings.
\end{cor}
A special case of Corollary \ref{Corollary: $F$-injectivity localizes}
was proved in \cite[Proposition 4.3]{SchwedeFinjAreDuBois} under the additional hypothesis of $F$-finiteness.

We now turn our attention to the non-local case. Let $R$ be a polynomial ring over a field
of prime characteristic $p$,
$\mathcal{I}\subseteq R$ an ideal, and write $S=R/\mathcal{I}$.
For each prime $P\subseteq R$, we obtain
local cohomology modules $\HH^{j}_{P \widehat{R_P}}(\widehat{S}_P)$,
where $\widehat{{\ }}$ \ denotes completion at the maximal ideal, each equipped with a natural Frobenius map as described in section \ref{Section: Natural Frobenius maps on local cohomology modules}.
For these Artinian $\widehat{R_P}[\Theta; f]$-modules
\begin{equation}\label{Equation: 1}
\Delta^e (\HH^{j}_{P \widehat{R_P}}(S_P)) = \Ext_{R}^{\dim R_P -j}(R/\mathcal{I}, R) \otimes \widehat{R_P} \rightarrow
\Ext_R^{\dim R_P -j}(R/\mathcal{I}^{[p]}, R) \otimes \widehat{R_P}
\end{equation}
and this map is  induced by the surjection
$R/\mathcal{I}^{[p]} \rightarrow R/\mathcal{I}$.

Now fix now presentations $\Ext_{R}^{i}(R/\mathcal{I}, R)=
\Coker \left(R^{\beta_i} \xrightarrow{A_i} R^{\alpha_i} \right)$
and let the maps $\Coker A_i \xrightarrow{U_i} \Coker A_i^{[p]}$ be isomorphic to the map
$$\Ext_{R}^{i}(R/\mathcal{I}, R) \rightarrow F_R^1 \Ext_{R}^{i}(R/\mathcal{I}, R) =
\Ext_{R}^{i}(R/\mathcal{I}^{[p]}, R)$$
induced by the surjection
$R/\mathcal{I}^{[p]} \rightarrow R/\mathcal{I}$.

For any prime $P\subseteq R$ containing $\mathcal{I}$
define $\eta(P,j)$ to be the HSL number of
$\HH^{\height P - j}_{P \widehat{R_P}}(S_P))$ equipped with its natural $\widehat{R_P}[\Theta; f]$-module structure.

\begin{thm}
For all $j\geq 0$ fix a presentation
$R^{\beta_j} \xrightarrow[]{A_j}  R^{\alpha_j} \rightarrow \Ext^j(R/\mathcal{I}, R)$
and an $\alpha_j \times \alpha_j$ matrix $U_j$ for which
$U_j : \Coker A_j \rightarrow \Coker A_j^{[p]}$ is isomorphic to the map
$\Ext_{R}^{j}(R/\mathcal{I}, R)  \rightarrow
\Ext_R^{j}(R/\mathcal{I}^{[p]}, R)$
induced by the surjection
$R/\mathcal{I}^{[p]} \rightarrow R/\mathcal{I}$.
For all $j,e\geq 0$ write $B_{j,e}=\Image A_j + I_e (\Image U_j^{[p^{e-1}]} \cdots U_j^{[p]} U_j)$.
Then $\eta(P,j)<e$ if and only if $P$ is not in the support of $B_{j,e-1}/B_{j,e}$.
\end{thm}

\begin{proof}
Note that $\Delta^1 ( \HH^{\height P - j}_{P \widehat{R_P}}(S_P) )$ is the completion at $P$ of
the map $\Coker A_j \xrightarrow[]{U_j} \Coker A_j^{[p]}$. The discussion above  of the local case now implies that
$\eta(P,j)<e$ if and only if
$$\frac
{\widehat{R_P} \otimes_R \Image A_j + I_{e-1} (\widehat{R_P} \otimes_R \Image U_j^{[p^{e-2}]} \cdots U_j^{[p]} U_j)}
{\widehat{R_P} \otimes_R \Image A_j + I_{e} (\widehat{R_P} \otimes_R \Image U_j^{[p^{e-1}]} \cdots U_j^{[p]} U_j)}
=0 .$$
But Lemma \ref{Lemma: commutation with localization and completion} shows that this quotient is
$$\frac
{\widehat{R_P} \otimes_R \Image A_j +{\widehat{R_P}} \otimes_R I_{e-1} (\Image U_j^{[p^{e-2}]} \cdots U_j^{[p]} U_j)}
{\widehat{R_P} \otimes_R \Image A_j + {\widehat{R_P}} \otimes_R I_e(\Image U_j^{[p^{e-1}]} \cdots U_j^{[p]} U_j)}
=
\widehat{R_P} \otimes_R (B_{e-1}/B_{e}) .$$
\end{proof}

\begin{cor}
The set $\{ \eta(P,j) \,|\, \mathcal{I} \subseteq P\subset R \text{ prime}, j\geq 0\}$ is bounded.
\end{cor}
\begin{proof}
Each set $U_{j,e} =\{ P\in \Spec R \,|\, \eta(P,j)<e \}$ is open, since
it is the complement of $\Supp(B_{j,e-1}/B_{j,e})$.
Also $\Spec R= \bigcup_{j\geq 0, e\geq 1} U_{j,e}$ hence the result follows
from the quasicompactness of $R$.
\end{proof}

\bibliographystyle{skalpha}
\bibliography{KatzmanBib}

\end{document}